 \documentclass[noinfoline,11pt]{imsart}

\usepackage{graphicx}
\usepackage{comment}
\usepackage{url} 

\usepackage{hyperref}

\usepackage[left=24mm, right=24mm, top=30mm, bottom=30mm]{geometry}

\usepackage{amssymb}
\usepackage{amsthm}

\usepackage{mathrsfs}
\usepackage{mathtools}
\usepackage{amsmath}
\usepackage{amsfonts}

\newcommand{\hilbert}{{\mathbb{H}}}
\newcommand{\randomel}{{L^2(\Omega;\mathbb{H})}}
\newcommand{\randomvar}{{L^2(\Omega;\mathbb{R})}}

\newcommand{\sigmafield}{\mathfrak{F}}

\newcommand{\Ex}{\mathbb{E}}
\newcommand{\s}{{\mathbb{S}^2}}
\newcommand{\hs}{\mathcal{S}_2}

\newcommand{\ls}{L^2(\s; \mathbb{R})}

\newcommand{\ylm}{Y_{\ell, m}}

\newcommand{\summ}{\sum_{m=-\ell}^{\ell}}

\newtheorem{theorem}{Theorem}
\newtheorem{condition}{Condition}

\newtheorem{proposition}{Proposition}
\newtheorem{remark}{Remark}
\newtheorem{definition}{Definition}
\newtheorem{lemma}{Lemma}
\newtheorem{corollary}{Corollary}

\begin{document}

\begin{frontmatter}

\title{Asymptotics for isotropic Hilbert-valued spherical random fields} 
\runtitle{Asymptotics for isotropic Hilbert-valued spherical random fields}

\begin{aug}
\author{\fnms{Alessia} \snm{Caponera}\thanksref{t1}\ead[label=e1]{alessia.caponera@epfl.ch}}

\runauthor{Caponera}

\affiliation{\'Ecole Polytechnique F\'ed\'erale de Lausanne}

\address{Institut de Math\'ematiques\\ \'Ecole Polytechnique F\'ed\'erale de Lausanne\\ \printead{e1} }

\thankstext{t1}{The author acknowledges support from SNSF Grant 200020\_207367. Moreover, the author wishes to thank Domenico Marinucci for many insightful discussions and suggestions.}

\end{aug}

\begin{abstract} 
In this paper, we introduce the concept of isotropic Hilbert-valued spherical random field, thus extending the notion of isotropic spherical random field to an infinite-dimensional setting. We then establish a spectral representation theorem and a functional Schoenberg's theorem. Following some key results established for the real-valued case, we prove consistency and quantitative central limit theorem for the sample power spectrum operators in the high-frequency regime.
\end{abstract}

\bigskip


\end{frontmatter}

\setcounter{tocdepth}{2}
\tableofcontents

\section{Introduction}

Spherical random fields have aroused a growing interest in the last decades due to their important connections with applications, for instance in Cosmology, Astrophysics, Earth and Climate Sciences -- see \cite[Chapter 1]{MP:11} for an overview.
Both the probabilistic and statistical aspects have been studied in many different directions. 

From a purely probabilistic point of view, there is a very rich literature that studies the geometry of isotropic spherical random fields, with particular attention to the asymptotic behavior of related geometric functionals -- see for instance \cite{MW:11, MW:14, MR:15, CM:18b, To:19, Ro:19, To:20, MRW:20, CGR:22} and the references therein.

Another vast research line focuses on covariance functions of spherical random fields, which under isotropy can be reduced to functions on the interval $[-1,1]$. 
I. J. Schoenberg in his seminal paper \cite{schoenberg} characterized the class of positive definite continuous function on $[-1, 1]$ and this result has received considerable attention in the last years, thanks also to the review \cite{gneiting2}. Indeed, it has had an important impact on both the probabilistic and statistical communities since it can be interpreted as the analogue of the Herglotz’s theorem for stationary time series (see for instance \cite{brockwelldavis}). Regularity properties have been investigated  for instance in \cite{Zi:14, Zi:17, langschwab}, while a natural extension of Schoenberg’s theorem to the product space $[-1, 1] \times G$, for $G$ a locally compact group, is given in \cite{bergporcu}. This last result clearly includes the case of isotropic in space and stationary in time covariance functions, which are motivated in particular by Earth and Climate Sciences. Parametric families in these settings have been also widely studied -- for a comprehensive review see \cite{gneiting2,PFN:20} and the references therein.

Very recently, the problem of investigating sequences of serially-dependent spherical random fields and their second-order structure has been also addressed using a functional time series approach in \cite{SPHARMA, cm19, cdv19, CFSP:22}. With functional time series analysis one usually refers to methods and techniques developed for collections of \emph{Hilbert-valued random variables} indexed by integers, where the index is interpreted as time -- see for instance \cite{Panaretos2,Panaretos}. In this framework, the assumption of spatial isotropy can be relaxed (see, e.g., \cite{CFSP:22}), while the hypothesis of temporal stationarity is usually preserved.

In this paper, we aim to adopt a \emph{reverse} point of view, in which the collection of Hilbert-valued random variables is indexed by points on the 2-dimensional sphere, thus extending the notion of isotropic spherical random field to an infinite-dimensional setting. Some real data examples which motivates this approach are presented in Section \ref{sec::zerone}.
Specifically, in Section \ref{sec::spectral-repr}, we introduce the concept of \emph{isotropic Hilbert-valued spherical random field} and \emph{power spectrum operator}. We then establish a spectral representation theorem and a functional Schoenberg's theorem. In Section \ref{sec::high-freq}, following some key results in \cite{MP:ergodicity, MP:11} established for the real-valued case, we prove consistency and \emph{quantitative} central limit theorem for the sample power spectrum operators in the \emph{high-frequency} regime. In Section \ref{sec::back} and in the Appendix we include some technical material.
\\

\section{Notation and background}\label{sec::back}

Let $\mathbb{K}$ be a separable real Hilbert space with norm $\|\cdot\|_{\mathbb{K}}$ and inner product $\langle \cdot, \cdot \rangle_{\mathbb{K}}$, $\mathfrak{B}(\mathbb{K})$ the Borel $\sigma$-field of $\mathbb{K}$.

Given a measure space $(\mathcal{X}, \mathfrak{A}, \mu)$ we define $L^2(\mathcal{X}, \mathfrak{A}, \mu; \mathbb{K})$ to be the Hilbert space of measurable functions $f: (\mathcal{X}, \mathfrak{A}) \to ( \mathbb{K}, \mathfrak{B}( \mathbb{K}))$ such that
$$
\int_{\mathcal{X}} \| f\|^2_{\mathbb{K}} d\mu < \infty.
$$
For simplicity, once the measure space is specified, we will write $L^2(\mathcal{X}; \mathbb{K})$.
When $\mathcal{X}$ is any subset of $\mathbb{R}^d, d \in \mathbb{N}$, we will implicitly consider the measure space with $\sigma$-field the Borel $\sigma$-field of $\mathcal{X}$, denoted by $\mathfrak{B}(\mathcal{X})$, and measure the Lebesgue measure.

If $(\Omega, \mathfrak{F}, \operatorname{P})$ is some probability space, a $\mathbb{K}$-valued random variable $F$ is a measurable
map from $(\Omega, \mathfrak{F})$ to $(\mathbb{K}, \mathfrak{B}(\mathbb{K}))$, and $F$  is Gaussian if, $\forall w \in \mathbb{K}$, $\langle F, w \rangle_\mathbb{K}$ is a real-valued Gaussian random variable.

\paragraph{Hermite polynomials and Wiener chaoses} For any integer $q \ge 0$, let denote by $H_q$ the $q$-th Hermite polynomial
\begin{equation*}
H_{q}(z)=(-1)^{q}e^{z^{2}/2}\frac{d^{q}}{dz^{q}}e^{-z^{2}/2}, \qquad z\in \mathbb{R};
\end{equation*}
for instance, the first few are given by $H_0(z) = 1, H_{1}(z)=z$, $H_{2}(z)=z^{2}-1$, $H_{3}(z)=z^{3}-3z$ and $H_{4}(z)=z^{4}-6z^{2}+3$.
Any finite variance transform of a standard Gaussian random variable $Z$ has a representation in terms of Hermite polynomials (see \cite[Example 2.2.6, p. 27]{noupebook}). This means that, given the probability space $(\Omega, \sigma(Z), \operatorname{P})$ and $\varphi \in L^2(\Omega; \mathbb{R})$, $\varphi$ can be written as an orthogonal series in $L^2(\Omega; \mathbb{R})$ as follows
\begin{equation*}
\varphi=\sum_{q=0}^{\infty }J_{q}(\varphi)\frac{H_{q}(Z)}{q!}, \qquad  J_{q}(\varphi)=\mathbb{E}[\varphi H_{q}(Z)].
\end{equation*}
More generally, given the probability space $(\Omega, \sigma(\mathcal{Z}), \operatorname{P})$, where $\mathcal{Z}$ is any $L^2(\Omega; \mathbb{R})$-closed centered linear Gaussian space, we can write the Stroock-Varadhan decomposition
\begin{equation*}
L^2(\Omega; \mathbb{R})=\bigoplus\limits_{q=0}^{\infty }\mathcal{H}_{q},
\end{equation*}
where $\mathcal{H}_{q}$ is the $q$-th order Wiener chaos, i.e., the space spanned by $q$-th order Hermite polynomials computed in the elements of $\mathcal{Z}$, see \cite[Chapter 2]{noupebook} for more discussions and details.
However, the notion of Wiener chaoses can be extended to general $\mathbb{K}$-valued random variables and it will be used in this paper to prove a \emph{quantitative functional central limit theorem} by means of the rich machinery of Stein-Malliavin methods (see \cite{campese,noupebook}).
To this purpose, we define the following distances between probability measures.
The $d_2$-distance between two $\mathbb{K}$-valued random variables $F$ and $G$  is defined by
$$
d_2(F,G) = \sup_{ \substack{h \in C_b^2(\mathbb{K}) \\ \|h\|_{C_b^2(\mathbb{K})} \le 1}} |\mathbb{E}[h(F)]- \mathbb{E}[h(G)]|,
$$
where $C_b^2(\mathbb{K})$ denotes the twice Fréchet differentiable, real-valued functions on $\mathbb{K}$ with uniformly bounded first and second derivatives (see \cite[Section 2.2.4 and Section 3]{campese}).
While, the total variation distance between two $\mathbb{K}$-valued random variables $F$ and $G$ is defined by
$$
d_{\operatorname{TV}}(F,G) = \sup_{B \in \mathfrak{B}(\mathbb{K})}  |\operatorname{P}[F \in B]- \operatorname{P}[G \in B]|.
$$

\paragraph{Harmonic analysis on the sphere}

Let $\ls$ be the Hilbert space of square-integrable real-valued functions over the unit sphere $\s = \{ x \in \mathbb{R}^3: \|x\|=1\}$. Through the paper, we shall use $\|\cdot\|$ and $\langle \cdot, \cdot \rangle$ to denote the Euclidean norm and inner product on $\mathbb{R}^3$ respectively.

It is well-known that $\ls$ can be decomposed into the direct sum of orthogonal spaces spanned by eigenfunctions of the corresponding Laplacian. More precisely, 
\begin{equation*}
\ls = \bigoplus_{\ell=0}^\infty \mathcal{Y}_\ell,
\end{equation*}
where $ \mathcal{Y}_\ell$ is spanned by the eigenfunctions of the spherical Laplacian 
$\Delta _{\mathbb{S}^{2}}$ associated with the eigenvalue $-\ell(\ell+1)$. These eigenfunctions are called spherical harmonics.
A standard orthonormal basis for the eigenspace $\mathcal{Y}_\ell$ is chosen to be the set of (in this paper, \emph{real-valued}) \emph{fully normalized spherical harmonics} $\{\ylm, \ m=-\ell,\dots,\ell\}$.

One of the most fundamental property of spherical harmonics (which will be also widely used through the paper) is expressed by the following \emph{addition formula}. For any $x,y \in \s$ and any integer $\ell \ge 0$,
\begin{equation*}\label{background::addition}
\summ \ylm(x) \ylm(y) = \frac{2\ell+1}{4\pi} P_{\ell}( \langle x, y \rangle ),
\end{equation*} 
where  $P_{\ell}:[-1,1] \to \mathbb{R}$ is the $\ell$-th Legendre polynomial defined by
\begin{align*}
P_\ell(z) = \frac{1}{2^\ell \ell !} \frac{d^\ell}{d z^\ell}(z^2-1)^\ell, \qquad z \in [-1,1].
\end{align*}
For a proof (based on group representation theory) we refer to \cite[Chapter 3]{MP:11}.

\begin{remark}
Legendre polynomials are orthogonal over $[-1,1]$, i.e.,
\begin{equation*}
\int_{-1}^{-1} P_\ell(z) P_{\ell'}(z) dz  = \frac{2}{2\ell+1} \delta_\ell^{\ell'},
\end{equation*}
here $\delta_a^b$ is the Kronecker delta function. Moreover, the sequence of Legendre polynomials forms an orthogonal basis for $L^2([-1,1]; \mathbb{R})$. Note that $P_\ell(1) = 1$, for all integers $\ell \ge 0$. See \cite{szego} for further details.
\end{remark}

\paragraph{\emph{Ad hoc} notation} In the next sections we will adopt the following notation.
Let $\hilbert$ be a separable real Hilbert space with norm $\|\cdot\|_\hilbert$ and inner product $\langle \cdot, \cdot \rangle_\hilbert$. Let $\mathcal{L}(\hilbert)$ be the set of all bounded linear operators from $\hilbert$ to $\hilbert$, and $\|\cdot\|_\infty$ be the corresponding operator norm.
 
For $1\le p <\infty$, we let $\|\cdot\|_p$ denote the $p$-Schatten norm, defined by
$$\|\mathscr{T}\|_p = \left ( \sum_{j=1}^\infty \langle (\mathscr{T}^* \!\mathscr{T} )^{p/2}  e_j, e_j \rangle_\hilbert\right )^{1/p}, \quad \mathscr{T} \text{ compact} \in \mathcal{L}(\hilbert),$$
for any complete orthonormal system (CONS) $\{e_j,\ j \in \mathbb{N}\}$.
Clearly, $\|\cdot\|_2, \|\cdot\|_1$ are the Hilbert-Schmidt and trace (or nuclear) norms, and if $\|\mathscr{T}\|_1<\infty$ we can also define the trace of $\mathscr{T}$ as
$$
\operatorname{Tr}(\mathscr{T}) = \sum_{j=1}^\infty \langle \mathscr{T}e_j, e_j \rangle_\hilbert.
$$
Moreover, we let $\mathcal{S}_p=\mathcal{S}_p(\hilbert)$ be the set of all compact operators $\mathscr{T} \in \mathcal{L}(\hilbert)$ such that $\|\mathscr{T}\|_p <\infty$. Recall that $\hs$ endowed with the Hilbert-Schmidt inner product, denoted by $\langle \cdot, \cdot \rangle_2$, is a Hilbert space, as well as $\hs^\dag$, the set of Hilbert-Schmidt operators that are also self-adjoint.

For $u,v \in \hilbert$, the tensor product $u \otimes v$ is defined to be the mapping that takes any element $f \in \hilbert$ to $u \langle f, v\rangle_\hilbert \in \hilbert.$
In particular, we shall use the notation $\otimes_2$ to indicate the tensor product on $\hs^\dag$.

\section{Isotropic $\hilbert$-valued spherical random fields: definition and spectral characteristics}\label{sec::spectral-repr}

Consider a collection $T=\{T(x), \ x \in \mathbb{S}^2\}$ of centered $\hilbert$-valued random variables defined on a common probability space $(\Omega, \sigmafield, \operatorname{P})$ and such that $\mathbb{E} \|T(x)\|^2_\hilbert < \infty$, for any $x \in \mathbb{S}^2$. Assume also that the mapping $T:\Omega \times \s \to \hilbert$ is \emph{jointly measurable}, i.e., measurable with respect to the product $\sigma$-field $\mathfrak{B}(\mathbb{S}^2) \times \sigmafield$. We call $T=\{T(x), \ x \in \mathbb{S}^2\}$ \emph{$\hilbert$-valued spherical random field}. 

Now, define as Bochner integrals the autocovariance operators
\begin{equation*}
\mathscr{R}_{x,y}  =  \mathbb{E} \left [ T(x) \otimes T(y) \right ] = \int_{\Omega} T(x) \otimes T(y) \, d  \operatorname{P}, \qquad x,y \in \mathbb{S}^2.
\end{equation*}

\begin{definition}
We say that the collection $T=\{T(x), \ x \in \mathbb{S}^2\}$ of zero-mean $\hilbert$-valued random variables is isotropic if $\mathbb{E} \|T(x) \|_\hilbert^2 < \infty$, for all $x \in \mathbb{S}^2$, and
\begin{equation*}
\mathscr{R}_{x,y} =\mathscr{R}_{\rho x, \rho y}, \qquad \text{for all } x,y \in \mathbb{S}^2, \, \rho \in SO(3),
\end{equation*}
where $SO(3)$ is the special group of rotations in $\mathbb{R}^3$.
\end{definition}
\begin{remark}
It is interesting to observe that, differently from the time-stationary case (see, e.g., \cite{Panaretos2}), such operators are self-adjoint even for $x\ne y$, i.e., $\mathscr{R}_{x,y}=\mathscr{R}_{y,x}$. Indeed, one can consider the axis in the direction of the midpoint along the geodesic path between $x$ and $y$, and a rotation about that axis.
Thus, in this case, as for standard isotropic random fields on the sphere, we can use the equivalent notation $\mathscr{R}_{\langle x, y \rangle}$. 
Lemma \ref{lemma::trace} in the Appendix then allows to write the following upper bound:
$$
\|\mathscr{R}_{\langle x, y \rangle} \|_1 \le \|\mathscr{R}_1 \|_1, \qquad \forall x,y \in \s.
$$
\end{remark}

\begin{remark}
Under joint measurability with the additional assumption of isotropy, by Fubini's theorem one has that 
$$\mathbb{E} \left [ \int_{\mathbb{S}^2 } \|T(x)\|^2_\hilbert d x  \right] = 4 \pi \mathbb{E} \|T(x_0) \|_\hilbert^2 < \infty,$$ 
for any fixed $x_0 \in \mathbb{S}^2$.
This implies that there exists a $\sigmafield$-measurable set $\Omega'$ of $\operatorname{P}$-probability 1 such that, for every $\omega \in \Omega'$,  $T(\omega, \cdot)$ is an element of $L^2(\mathbb{S}^2; \hilbert)$.
\end{remark}

In the following we formally summarize the assumptions we need for our results.
\begin{condition}\label{cond::1}
$T=\{T(x), \ x \in \mathbb{S}^2\}$ is a centered $\hilbert$-valued spherical random field that is also isotropic with autocovariance operators $\mathscr{R}_{\langle \cdot, \cdot \rangle} $.
\end{condition}


Now, we are able to define the infinite-dimensional equivalent of the power spectrum of real-valued spherical random fields. This infinite-dimensional power spectrum will play a fundamental role in the rest of the paper. It will be key indeed in the derivation of the spectral representation theorem (Theorem \ref{theo::spectral-sphere}) and in the characterization of the covariance operators (Theorem \ref{theo::schoenberg}), presented in this section. Moreover, in Section \ref{sec::high-freq}, we will provide asymptotic results for its sample counterpart.

\begin{proposition}[Power spectrum operator]\label{prop::ps}
Under Condition \ref{cond::1}, for any integer $\ell \ge 0$, the power spectrum operator $\mathscr{F}_\ell: \hilbert \to \hilbert$ defined by
$$
\mathscr{F}_\ell = \int_{\s \times \s}  \mathscr{R}_{\langle x, y\rangle} Y_{\ell,0}(x)  Y_{\ell,0}(y)dx dy  
$$
is sefl-adjoint, nonnegative definite and nuclear.
\end{proposition}

\begin{remark}
We will denote by $\mathcal{C}_\ell$ the nuclear norm of $\mathscr{F}_\ell$, which is given by
$$\mathcal{C}_\ell = \|\mathscr{F}_\ell\|_1= \int_{\s \times \s} \operatorname{Tr}( \mathscr{R}_{\langle x, y\rangle} )Y_{\ell,0}(x)  Y_{\ell,0}(y)dx dy.$$ 
The sequence of nuclear norms $\{\mathcal{C}_\ell, \ell \ge 0\}$ will be called \emph{reduced power spectrum}.
\end{remark}

\begin{proof}[Proof of Proposition \ref{prop::ps}]
First note that, for any $(\ell,m)$ and $(\ell',m')$, the Bochner integral
$$
\int_{\s \times \s}  \mathscr{R}_{\langle x, y\rangle} Y_{\ell,m}(x)   Y_{\ell',m'}(y)dx dy
$$
is well defined on $\hs$. Indeed, 
$$
\int_{\s \times \s}  \|\mathscr{R}_{\langle x, y\rangle}\|_2 |Y_{\ell,m}(x)|  | Y_{\ell',m'}(y)| dx dy \le4\pi \|\mathscr{R}_1\|_1 < \infty.
$$
Moreover, for $f \in \hilbert$,
$$
\left ( \int_{\s \times \s}  \mathscr{R}_{\langle x, y\rangle} Y_{\ell,m}(x)  Y_{\ell',m'}(y)dx dy  \right) f =\int_{\s \times \s}  (\mathscr{R}_{\langle x, y\rangle}f )Y_{\ell,m}(x)  Y_{\ell',m'}(y)dx dy
$$
(see \cite[Example 4.4.6]{Hsing}), and since the inner product is continuous
$$
\left \langle \int_{\s \times \s}  (\mathscr{R}_{\langle x, y\rangle}f ) Y_{\ell,m}(x)  Y_{\ell',m'}(y)dx dy , g \right \rangle_\hilbert= \int_{\s \times \s}  \langle \mathscr{R}_{\langle x, y\rangle}f, g\rangle_\hilbert Y_{\ell,m}(x)  Y_{\ell',m'}(y)dx dy,
$$
for all $f,g \in \hilbert$. 

Define
$$
\mathscr{F}_\ell = \int_{\s \times \s}  \mathscr{R}_{\langle x, y\rangle} Y_{\ell,0}(x)  Y_{\ell,0}(y)dx dy.
$$
We can prove that, for any $(\ell,m)$ and $(\ell',m')$,
\begin{equation}\label{eq::uncorr}
\int_{\s \times \s}  \mathscr{R}_{\langle x, y\rangle} Y_{\ell,m}(x)  Y_{\ell',m'}(y)dx dy =
\begin{cases}
\mathscr{F}_\ell &\ell=\ell' \text{ and } m=m'\\
0 &\text{otherwise}
\end{cases}.
\end{equation}
Indeed, for any $f,g\in \hilbert$, the function $(x,y) \mapsto \langle \mathscr{R}_{\langle x, y\rangle}f, g\rangle_\hilbert$ belongs to $L^2(\s \times \s; \mathbb{R})$ and
$$
\int_{\s \times \s}  \langle \mathscr{R}_{\langle x, y\rangle}f, g\rangle_\hilbert Y_{\ell,m}(x)  Y_{\ell',m'}(y)dx dy = \begin{cases}
\int_{\s \times \s}  \langle \mathscr{R}_{\langle x, y\rangle}f, g\rangle_\hilbert Y_{\ell,0}(x)  Y_{\ell,0}(y)dx dy &\ell=\ell' \text{ and } m=m'\\
0 &\text{otherwise}
\end{cases}.
$$

For any integer $\ell\ge 0$, the operator $\mathscr{F}_\ell$ is self-adjoint since $\mathscr{R}_{\langle x, y\rangle}=\mathscr{R}_{\langle y, x\rangle}$, and nonnegative definite since
\begin{align*}
\int_{\s \times \s}  \mathscr{R}_{\langle x, y\rangle} Y_{\ell,0}(x)  Y_{\ell,0}(y)dx dy &= \mathbb{E} \left[ \int_{\s \times \s}  T(x) \otimes T(y) Y_{\ell,0}(x)  Y_{\ell,0}(y) dx dy \right]\\
&= \mathbb{E}\left[\int_\s T(x) Y_{\ell,0}(x) dx  \otimes \int_\s T(y)  Y_{\ell,0}(y)  dy \right],
\end{align*}
which is a covariance operator. Moreover, $\mathscr{F}_\ell$ is nuclear since
$$
\|\mathscr{F}_\ell\|_1 \le \int_{\s \times \s}  \|\mathscr{R}_{\langle x, y\rangle}\|_1 |Y_{\ell,0}(x)|  | Y_{\ell,0}(y)| dx dy \le 4\pi \|\mathscr{R}_1\|_1 < \infty.
$$
\end{proof}

As for the real-valued case \cite[Theorem 5.13]{MP:11}, we are able to derive a crucial result, that is, a spectral representation theorem for isotropic $\hilbert$-valued spherical random fields. 

\begin{theorem}[Spectral representation theorem]\label{theo::spectral-sphere}
Let $T=\{T(x), \ x \in \s\}$ as defined in Condition \ref{cond::1}. Then, it admits the representation
\begin{equation}\label{cramer}
T(x) = \sum_{\ell=0}^\infty \sum_{m=-\ell}^{\ell} a_{\ell,m}  Y_{\ell, m}(x),
\end{equation}
where, for fixed $(\ell, m)$, $a_{\ell,m}$ is a zero-mean $\hilbert$-valued random variable with $\mathbb{E}\left [a_{\ell,m} \otimes a_{\ell,m} \right] = \mathscr{F}_\ell$ and
\begin{equation*}
\mathbb{E}\left [a_{\ell,m} \otimes a_{\ell',m'} \right] = 0, \qquad \text{for } \ell \ne \ell', \, m \ne m'.
\end{equation*}
The convergence of the infinite series in \eqref{cramer} is both in the sense of $L^2(\Omega \times \s ; \hilbert)$ and $L^2(\Omega ; \hilbert)$ for every fixed $x \in \s$, that is,
\begin{equation}\label{expansion-spherel2}
\mathbb{E}\int_\s \left \| T(x) - \sum_{\ell=0}^L \sum_{m=-\ell}^{\ell} a_{\ell,m}  Y_{\ell, m}(x) \right \|^2_\hilbert dx \to 0, \qquad L \to \infty,
\end{equation}
and for every fixed $x \in \s$
\begin{equation}\label{expansion-sphere}
\mathbb{E}\left \| T(x) - \sum_{\ell=0}^L \sum_{m=-\ell}^{\ell} a_{\ell,m}  Y_{\ell, m}(x) \right \|^2_\hilbert \to 0, \qquad L \to \infty.
\end{equation}
\end{theorem}

\begin{proof} [Proof of Theorem \ref{theo::spectral-sphere}]
We first note that, for any $(\ell,m)$, we can define the $\hilbert$-valued random variable $a_{\ell,m}$ as the Bochner integral
$$
a_{\ell,m} = \int_\s T(x) Y_{\ell,m}(x) dx;
$$
indeed 
$$
\int_\s \|T(x)\|_\hilbert |Y_{\ell,m}(x)| dx \le  \left (\int_\s \|T(x)\|_\hilbert^2\, dx \right)^{1/2} < \infty.
$$
Moreover, given any $f \in \hilbert$,
$$
\langle a_{\ell, m}, f \rangle_\hilbert  = \int_\s  \langle T(x), f \rangle_\hilbert Y_{\ell,m}(x) dx ,
$$
by continuity of inner product and properties of Bochner integrals.

Clearly, $\mathbb{E}[a_{\ell,m}] = 0$ and
$$
\mathbb{E}[ a_{\ell, m} \otimes  a_{\ell', m'} ] = \int_{\s \times \s}  \mathscr{R}_{\langle x, y\rangle} Y_{\ell,m}(x)  Y_{\ell',m'}(y)dx dy =
\begin{cases}
\mathscr{F}_\ell &\ell=\ell' \text{ and } m=m'\\
0 &\text{otherwise}
\end{cases},
$$
see Equation \eqref{eq::uncorr}.

Consider now any complete orthonormal system $\{e_j, \  j\in \mathbb{N}\}$ of $\hilbert$ and observe that, 
\begin{align*}
&\mathbb{E} \left[ \int_\s \left \| T(x) - \sum_{\ell=0}^L \sum_{m=-\ell}^{\ell} a_{\ell,m}  Y_{\ell, m}(x) \right \|^2_\hilbert dx  \right] \\
=& \mathbb{E}\left[ \sum_{j=1}^\infty \int_\s\left( \langle T(x), e_j \rangle_\hilbert - \sum_{\ell=0}^L \sum_{m=-\ell}^\ell \langle a_{\ell, m}, e_j \rangle_\hilbert Y_{\ell,m}(x) \right)^2 dx  \right],
\end{align*}
with
$$
\langle a_{\ell, m}, e_j \rangle_\hilbert  = \int_\s  \langle T(x), e_j \rangle_\hilbert Y_{\ell,m}(x) dx .
$$
For any $j \in \mathbb{N}$, almost surely it holds that
$$
\int_\s \langle T(x), e_j \rangle^2_\hilbert \, dx \le \int_\s \|T(x) \|^2_\hilbert \, dx < \infty,
$$
thus $\langle T(x), e_j \rangle_\hilbert$ is almost surely in $L^2(\s; \mathbb{R})$. As a consequence, 
\begin{equation}\label{ej-conv}
\int_\s\left( \langle T(x), e_j \rangle_\hilbert - \sum_{\ell=0}^L \sum_{m=-\ell}^\ell \langle a_{\ell, m}, e_j \rangle_\hilbert Y_{\ell,m}(x) \right)^2 dx \to 0, \qquad L\to \infty.
\end{equation}
From \eqref{ej-conv}, we can assert that
\begin{align*}
\int_\s\left( \langle T(x), e_j \rangle_\hilbert - \sum_{\ell=0}^L \sum_{m=-\ell}^\ell \langle a_{\ell, m}, e_j \rangle_\hilbert Y_{\ell,m}(x) \right)^2 dx &= \sum_{\ell=L+1}^\infty \langle a_{\ell, m}, e_j \rangle_\hilbert^2\\
&\le  \sum_{\ell=0}^\infty \langle a_{\ell, m}, e_j \rangle_\hilbert^2\\
& = \int_{\s} \langle T(x), e_j \rangle_\hilbert^2 \, dx < \infty,
\end{align*}
and since
$$
\mathbb{E}\left [\sum_{j=1}^\infty \int_{\s} \langle T(x), e_j \rangle_\hilbert^2 \, dx \right]  = \mathbb{E} \left [\int_\s \|T(x)\|_\hilbert^2\, dx  \right]< \infty,
$$
the proof of \eqref{expansion-spherel2} is concluded by using dominated convergence.

Now, we prove that
$$
\mathbb{E}\left  \| T(x) - \sum_{\ell=0}^L \sum_{m=-\ell}^\ell a_{\ell, m} Y_{\ell,m}(x) \right\|^2_\hilbert
$$
is constant with respect to $x$, by action of isotropy.
We have that
\begin{align*}
&\mathbb{E}\left  \| T(x) - \sum_{\ell=0}^L \sum_{m=-\ell}^\ell a_{\ell, m} Y_{\ell,m}(x) \right\|^2_\hilbert  \\ 
=\,& \mathbb{E}  \| T(x) \|^2_\hilbert + \sum_{\ell=0}^L \sum_{m=-\ell}^\ell \sum_{\ell'=0}^L \sum_{m'=-\ell'}^{\ell'} \mathbb{E}\langle a_{\ell, m}, a_{\ell',m'} \rangle_\hilbert Y_{\ell,m}(x) Y_{\ell',m'}(x) \\
-\, & 2 \sum_{\ell=0}^L \sum_{m=-\ell}^\ell \langle T(x), a_{\ell,m} \rangle_\hilbert Y_{\ell,m}(x)\\
=\,&\|\mathscr{R}_1\|_1 - \sum_{\ell=0}^L  \mathcal{C}_\ell \frac{2\ell+1}{4\pi}.
 \end{align*}
Hence,
$$
\mathbb{E}\left  \| T(x) - \sum_{\ell=0}^L \sum_{m=-\ell}^\ell a_{\ell, m} Y_{\ell,m}(x) \right\|^2_\hilbert = \frac{1}{4\pi}\mathbb{E}  \int_\s \left  \| T(y) - \sum_{\ell=0}^L \sum_{m=-\ell}^\ell a_{\ell, m} Y_{\ell,m}(y) \right\|^2_\hilbert dy,
$$
so that the conclusion is immediately deduced from \eqref{expansion-spherel2}.
Note that
\begin{equation}\label{sumCl}
\mathbb{E}  \|T(x)\|_\hilbert^2  = \|\mathscr{R}_1\|_{1}= \sum_{\ell=0}^\infty \mathcal{C_\ell} \frac{2\ell+1}{4\pi}  < \infty.
\end{equation}
\end{proof}

Now, in the spirit of \cite[Theorem 5.1]{marinucci2013mean} and \cite[Proposition 2]{SPHARMA}, we prove continuity of $\mathscr{R}_{\langle \cdot,\cdot \rangle}$ on $\s \times \s$ and we show that an expansion in terms of the sequence of power spectrum operators holds uniformly.

\begin{theorem}[Functional Schoenberg's theorem]\label{theo::schoenberg}
Under Condition \ref{cond::1}, the mapping $(x,y) \mapsto  \mathscr{R}_{\langle x,y \rangle}$ is continuous on $\s \times \s$, i.e.,  
$$
(x_n, y_n) \to (x_0,y_0) \implies \|\mathscr{R}_{\langle x_n,y_n\rangle} - \mathscr{R}_{\langle x_0, y_0\rangle}\|_1 \to 0, 
$$
and 
$$
\sup_{(x,y) \in \s \times \s} \left \| \mathscr{R}_{\langle x,y \rangle} -  \sum_{\ell=0}^L \mathscr{F}_{\ell} \frac{2\ell+1}{4\pi} P_\ell(\langle x, y\rangle)\right \|_1 \to 0, \qquad L\to\infty.
$$
\end{theorem}

\begin{corollary}[Mean-square continuity]\label{prop::continuity} $T=\{T(x), \ x \in \s \}$ is mean-square continuous, i.e.,
\begin{equation}\label{func-meansquare} 
\lim_{x \to x_0} \mathbb{E}\| T(x) - T(x_0) \|_\hilbert^2 = 0, \qquad \forall x_0 \in \mathbb{S}^2.
\end{equation}
Note that $\mathbb{E}\| T(x) - T(x_0) \|_\hilbert^2 = \| \mathbb{E} (T(x) - T(x_0) ) \otimes ( T(x) - T(x_0) ) \|_1$ (convergence in trace-norm to the zero operator).
\end{corollary}
This can be deduced from continuity by observing that
\begin{align*}
\Ex \|T(x) - T(x_0)\|^2_\hilbert 
&= 2 \|\mathscr{R}_1\|_1 - 2 \operatorname{Tr}(\mathscr{R}_{\langle x, x_0 \rangle}).
\end{align*}

\begin{proof}[Proof of Theorem \ref{theo::schoenberg}]

Consider the space of continuous functions from $\s \times \s$ to $\mathcal{S}_1$, written $C_b(\s \times \s; \mathcal{S}_1)$.
We can immediately see that, for any $L\in \mathbb{N}$,
$$
\sum_{\ell=0}^L \mathscr{F}_\ell \frac{2\ell+1}{4\pi} P_\ell(\langle x, y\rangle) \in C_b(\s \times \s; \mathcal{S}_1),
$$
by continuity of Legendre polynomials and $\langle \cdot, \cdot \rangle$. Moreover, for $L>L'$,
$$
\sup_{(x,y) \in \s \times \s }\left \| \sum_{\ell=L'+1}^L \mathscr{F}_\ell \frac{2\ell+1}{4\pi} P_\ell(\langle x, y\rangle) \right\|_1  \le \sum_{\ell=L'+1}^\infty \|\mathscr{F}_\ell\|_1 \frac{2\ell+1}{4\pi},
$$
and the sequence of partial sums is Cauchy; here we used the fact that Legendre polynomials are bounded by 1 and \eqref{sumCl}. Since the space of nuclear operators $\mathcal{S}_1$ equipped with the nuclear norm is a Banach space, then  $C_b(\s \times \s; \mathcal{S}_1)$ is also Banach under the uniform norm. As a consequence, this sum must converge to a continuous function.
We now prove that the limit is actually $\mathscr{R}_{\langle x, y \rangle}$. 

We will use the following characterization of the trace norm: if $\mathscr{T} \in \mathcal{S}_1$, then
$$\|\mathscr{T} \|_1 = \sup_{D \in \mathcal{D}}|\operatorname{Tr}(D\mathscr{T})|,$$
where $\mathcal{D}=\{D:\hilbert \to \hilbert \text{ compact}, \|D\|_\infty \le 1 \}.$

We first show that, for any $D \in \mathcal{D}$ and $f,g \in \hilbert$,
\begin{equation}\label{Cop-conv}
\left | \langle D\mathscr{R}_{\langle x,y \rangle} f, g \rangle_\hilbert  - \sum_{\ell=0}^L \langle D\mathscr{F}_\ell f, g \rangle_\hilbert   \frac{2\ell+1}{4\pi} P_\ell(\langle x, y\rangle) \right | \to 0,
\end{equation}
pointwise for any $x,y \in \s$.
Define the partial sum $$T_L(x) = \sum_{\ell=0}^L \sum_{m=-\ell}^{\ell}a_{\ell,m} Y_{\ell,m}(x),\qquad x \in \mathbb{S}^2,$$ with $a_{\ell,m}  = \int_{\mathbb{S}^2} T(x)Y_{\ell,m}(x) dx$.
By properties of Bochner integrals, we have that
\begin{align*}
&\langle D\mathscr{R}_{\langle x,y \rangle} f, g \rangle_\hilbert  - \sum_{\ell=0}^L \langle D\mathscr{F}_\ell f, g \rangle_\hilbert   \frac{2\ell+1}{4\pi} P_\ell(\langle x, y\rangle) \\
=&\, \mathbb{E}[ \langle DT(x), g \rangle_\hilbert  \langle  T(y), f \rangle_\hilbert]    -  \mathbb{E}[ \langle DT_L(x), g \rangle_\hilbert  \langle  T_L(y), f \rangle_\hilbert  .
\end{align*}
Now, by adding and subtracting the term
$
 \mathbb{E}[ \langle DT_L(x), g \rangle_\hilbert  \langle  T(y), f \rangle_\hilbert ],
$
we can observe that 
\begin{align*}
&\left | \mathbb{E}[ \langle DT(x), g \rangle_\hilbert  \langle  T(y), f \rangle_\hilbert -  \mathbb{E}[ \langle DT_L(x), g \rangle_\hilbert  \langle  T(y), f \rangle_\hilbert\right | \\
\le & \, \|f\|_\hilbert \|g\|_\hilbert \sqrt{\mathbb{E}\|T(y)\|_\hilbert^2}  \sqrt{\mathbb{E}\|T(x) - T_L(x)\|_\hilbert^2},
\end{align*}
while
\begin{align*}
&\left | \mathbb{E}[ \langle DT_L(x), g \rangle_\hilbert  \langle  T(y), f \rangle_\hilbert -  \mathbb{E}[ \langle DT_L(x), g \rangle_\hilbert  \langle  T_L(y), f \rangle_\hilbert\right | \\
\le & \, \|f\|_\hilbert \|g\|_\hilbert \sqrt{\mathbb{E}\|T_L(x)\|_\hilbert^2}  \sqrt{\mathbb{E}\|T(y) - T_L(y)\|_\hilbert^2}.
\end{align*}
From Theorem \ref{theo::spectral-sphere}, we have that 
\begin{equation*}
\lim_{L\to \infty} \mathbb{E}\|T(x) - T_L(x)\|^2_\hilbert=0,
\end{equation*}
and $\mathbb{E}\|T_L(x)\|_\hilbert^2 \le \mathbb{E}\|T(x)\|_\hilbert^2 < \infty$, for every $x \in \mathbb{S}^2$. Thus, \eqref{Cop-conv} is proved.

Lastly, we note that
\begin{align*}
&\left \| \mathscr{R}_{\langle x,y \rangle} - \sum_{\ell=0}^L \mathscr{F}_{\ell} \frac{2\ell+1}{4\pi} P_\ell(\langle x, y\rangle) \right \|_1  =  \sup_{D \in \mathcal{D}} \left | \operatorname{Tr} \left (D \mathscr{R}_{\langle x,y \rangle} - \sum_{\ell=0}^L D\mathscr{F}_{\ell} \frac{2\ell+1}{4\pi} P_\ell(\langle x, y\rangle) \right) \right|\\
=& \sup_{D \in \mathcal{D}} \left | \sum_{j=1}^\infty \left \langle \left (D\mathscr{R}_{\langle x,y \rangle} - \sum_{\ell=0}^L D\mathscr{F}_{\ell} \frac{2\ell+1}{4\pi} P_\ell(\langle x, y\rangle) \right) e_j, e_{j}\right \rangle_\hilbert \right |\\
\le& \sup_{D \in \mathcal{D}}  \sum_{j=1}^\infty \sum_{\ell=L+1}^\infty | \langle D \mathscr{F}_\ell e_j, e_{j} \rangle_\hilbert | \frac{2\ell+1}{4\pi}  \\
\le& \sup_{D \in \mathcal{D}}  \sum_{\ell=L+1}^\infty  \|D\mathscr{F}_\ell\|_1  \frac{2\ell+1}{4\pi} \\
\le&  \sum_{\ell=L+1}^\infty  \|\mathscr{F}_\ell\|_1  \frac{2\ell+1}{4\pi},
\end{align*}
thus
\begin{align*}
\sup_{(x,y) \in \s \times \s}\left \| \mathscr{R}_{\langle x,y \rangle} - \sum_{\ell=0}^L \mathscr{F}_{\ell} \frac{2\ell+1}{4\pi} P_\ell(\langle x, y\rangle) \right \|_1 \le  \sum_{\ell=L+1}^\infty  \|\mathscr{F}_\ell\|_1  \frac{2\ell+1}{4\pi} \to 0,
\end{align*}
and the proof is concluded.
\end{proof}

\subsection{The $L^2([0,1]; \mathbb{R})$ case}\label{sec::zerone}

As an example, we specialize the results to the space $L^2([0,1]; \mathbb{R})$, which can be of great interest for real applications. Indeed, in this case, for each point on the surface of the sphere we can think of observing some real-valued function defined on a given interval, such as the concentrations of particulate matter in the Earth's atmosphere or the air temperature and pressure at different altitudes. However, the interval $[0,1]$ can be even interpreted as a time interval (without the necessity of imposing temporal stationarity).
This framework can also be suited in Cosmology and Astrophysics for the study the Cosmic Microwave Background (CMB) from a functional data analysis perspective, e.g., by a model of the form
$$
T(x, u) = \mu(x) + Z(x, u),
$$
where $ x \in \mathbb{S}^2$ is a point of the sky, $u \in [0,1]$ represents a \emph{rescaled} frequency, $\mu(\cdot)$ is a temperature function and $Z$ is the contribution from foregrounds and instrumental noise to the map
$T$ (see, e.g., \cite{DRD:astro}). For instance, the \emph{Planck} (\url{https://www.esa.int/Planck}) instruments were able to measure sky radiation at different frequencies, specifically in the frequency range of 27 GHz to 1 THz. Such data can be used to extract the most accurate estimates of the spatial variations of the temperature of the CMB radiation, a key to understanding the origin of the Universe and the evolution of galaxies.
\\

Consider the collection of centered real-valued random variables $$\{T(x, u), \ x \in \mathbb{S}^2, u \in [0,1]\}$$ that is measurable with respect to the product $\sigma$-field $\mathfrak{B}(\mathbb{S}^2) \times \mathfrak{B}([0,1]) \times \sigmafield$. We assume that, for all $x \in \mathbb{S}^2$, $T(x, \cdot)$ belongs almost surely to $\hilbert= L^2([0,1]; \mathbb{R})$, and such collection of $L^2([0,1]; \mathbb{R})$-valued random variables satisfies Condition \ref{cond::1}.

Observe that, if $\mathbb{E} \| T(x,\cdot) \|_\hilbert^2  < \infty$, by Fubini's theorem $\mathbb{E} |T(x, u)|^2 < \infty$ $u$-almost everywhere. Then, the autocovariance kernel 
$
r_{\langle x,y \rangle}(\cdot, \cdot): (u, v) \mapsto \mathbb{E}\left [ T(x, u) T(y, v) \right]$
is in $L^2([0,1] \times [0,1]; \mathbb{R})$, and the corresponding operator $\mathcal{R}_{\langle x,y \rangle}: \hilbert \to \hilbert$ induced by right integration
$$(\mathcal{R}_{\langle x,y \rangle} f)(\cdot) = \int_0^1 r_{\langle x,y \rangle}(\cdot,v) f(v) d v$$ 
coincides with the autocovariance operator $\mathscr{R}_{\langle x,y \rangle}$, see also \cite{Panaretos2}. We shall use the notation $\mathscr{R}_{\langle x,y \rangle}$ for both.
It is readily seen that
\begin{equation*}
(\mathscr{F}_\ell \, g)(\cdot) = \int_0^1 C_\ell (\cdot, v) g(v) dv, \qquad \forall g \in L^2([0,1]; \mathbb{R}),
\end{equation*}
with
\begin{equation*}
C_\ell (u, v) = \int_{\mathbb{S}^2 \times \mathbb{S}^2} r_{\langle x, y\rangle}(u,v) Y_{\ell,0}(x) Y_{\ell, 0}(y) dx dy,
\end{equation*}
defined as an element of $L^2([0,1]\times[0,1]; \mathbb{R})$.
Furthermore, it holds that
\begin{equation*}
\sup_{(x,y) \in \s \times \s} \int_{[0,1]\times[0,1]} \left  | r_{\langle x, y\rangle}(u,v) - \sum_{\ell=0}^L C_\ell(u,v) \frac{2\ell+1}{4\pi} P_\ell(\langle x, y \rangle)\right |^2 dudv \to 0, \qquad L\to \infty.
\end{equation*}

If in addition, for every $x \in \mathbb{S}^2$, $\{T(x,u), \ u \in [0,1]\}$ is mean-square continuous, or equivalently $r_1(\cdot, \cdot)$ is continuous on $[0,1] \times [0,1]$, we can rephrase everything in a pointwise sense.
Moreover, due to isotropy, this further assumption implies joint mean-square continuity, i.e.,
$$
\lim_{(x,u) \to(x_0,u_0)} \mathbb{E}|T(x,u)-T(x_0,u_0)|^2=0, \qquad \forall (x_0,u_0) \in \mathbb{S}^2 \times[0,1].
$$

\section{High-frequency asymptotics}\label{sec::high-freq}

In this section, we consider $T=\{T(x), x \in \s\}$ satisfying Condition \ref{cond::1} that is also $Gaussian$, i.e., a collection of jointly Gaussian $\hilbert$-valued random variables. This means that for any $n \in \mathbb{N}$, any $x_1,\dots,x_n \in \s$, any $f_1,\dots, f_n \in \hilbert$, 
$$
\sum_{i=1}^n \langle T(x_i), f_i \rangle_\hilbert
$$
is a real-valued Gaussian random variable. Then, the $a_{\ell, m}$'s are (independent) Gaussian $\hilbert$-valued random variables. The converse is also true due to Theorem \ref{theo::spectral-sphere}.

In general,
each $a_{\ell,m}$ can be decomposed into the sum of uncorrelated components by means of its principal component decomposition, i.e., if $\{e_{j;\ell}, j\in \mathbb{N}\}$ are the eigenvectors of $\mathscr{F}_\ell$ then
$$
a_{\ell,m} = \sum_{j=1}^\infty \langle a_{\ell,m}, e_{j;\ell} \rangle_\hilbert\,  e_{j;\ell},
$$
with probability one, where $\{ \langle a_{\ell,m},  e_{j;\ell} \rangle_\hilbert $, $j\in \mathbb{N}\}$ are uncorrelated random variables with mean zero and variances the corresponding eigenvalues of $\mathscr{F}_\ell$, say $\{\lambda_{j;\ell}, j\in \mathbb{N}\}$. As a consequence, if $a_{\ell,m}$ is a Gaussian $\hilbert$-valued random variable, then $\langle a_{\ell,m},  e_{j;\ell} \rangle_\hilbert \sim \mathcal{N}(0, \lambda_{j;\ell})$.

Without loss of generality, we assume that $\|\mathscr{F}_\ell\|_\infty \ne 0$, for any integer $\ell \ge 0$. So that the $a_{\ell, m}$'s are nontrivial random variables. We then define the sample power spectrum operator as
$$
\hat{\mathscr{F}}_\ell  = \frac{1}{2\ell+1} \sum_{m=-\ell}^\ell a_{\ell,m} \otimes a_{\ell,m}.
$$

\begin{remark}
In the $L^2$ case, the sample power spectrum operator takes the form
$$
\hat{\mathscr{F}}_\ell g  = \int_0^1 \hat{C_\ell} (\cdot, v) g(v) dv, \qquad \forall g \in L^2([0,1]; \mathbb{R}),
$$
with
$$
\hat{C_\ell} (u,v) = \frac{1}{2\ell+1} \sum_{m=-\ell}^\ell a_{\ell,m}(u) a_{\ell,m} (v),
$$
defined almost surely as an element of $L^2([0,1]\times[0,1]; \mathbb{R})$. 
\end{remark}

One can immediately observe that
\begin{align*}
\mathbb{E}\| \hat{\mathscr{F}}_\ell - \mathscr{F}_\ell \|_{1} &\le \mathbb{E}\| a_{\ell,m} \otimes a_{\ell,m}  \|_{1}  +\| \mathscr{F}_\ell \|_{1} \\
&= \mathbb{E} \|a_{\ell,m} \|_\hilbert^2 + \| \mathscr{F}_\ell \|_{1} \\
&= 2  \| \mathscr{F}_\ell \|_{1},
\end{align*}
implying that $\| \hat{\mathscr{F}}_\ell - \mathscr{F}_\ell \|_{p}  \to 0$ in probability, as $\ell \to \infty,$ for all $1 \le p \le \infty$.
Moreover, for any $\epsilon >0$,
\begin{align*}
\sum_{\ell=0}^\infty \operatorname{P} \left (\| \hat{\mathscr{F}}_\ell - \mathscr{F}_\ell \|_{1} >\epsilon \right)
&\le \frac2\epsilon\sum_{\ell=0}^\infty \|\mathscr{F}_\ell \|_1 < \infty,
\end{align*}
which proves (by Borell-Cantelli Lemma) almost sure convergence, again in any $p$-norm.

However, it becomes nontrivial to ask whether a proper \emph{standardization} will give as well convergence results in some $p$-norms. This will allow to obtain a more complete information on the asymptotic performance of $\hat{\mathscr{F}}_\ell$ (and hence of any statistical procedure based on $\hat{\mathscr{F}}_\ell$). In line with the real-valued case (see \cite{MP:ergodicity} and \cite[Section 8.2]{MP:11}), we will refer to the following result as \emph{high frequency ergodicity}.

\begin{lemma}\label{lemma::second-moment}
For all integers $\ell \ge 0$, it holds that
\begin{equation*}
\mathbb{E}\| (\|\mathscr{F}_\ell\|_2^2+\mathcal{C}_\ell^2)^{-1/2}(\hat{\mathscr{F}}_\ell - \mathscr{F}_\ell )\|^2_{2} = \frac{1}{2\ell+1}=o(1).
\end{equation*}
\end{lemma}
The normalization term in Lemma \ref{lemma::second-moment} differs from the real-valued case. Indeed, the real-valued case can be seen as a particular case of the Hilbert-valued one, with $\|\mathscr{F}_\ell\|_2 = \mathcal{C}_\ell$.
\begin{proof}[Proof of Lemma \ref{lemma::second-moment}]
The rate can be obtained by first observing that
\begin{align*}
\mathbb{E}\|\hat{\mathscr{F}}_\ell - \mathscr{F}_\ell \|^2_{2} &= \mathbb{E}\left \| \frac{1}{2\ell+1} \sum_{m=-\ell}^\ell \left(a_{\ell, m} \otimes a_{\ell, m}  - \mathscr{F}_\ell \right) \right \|^2_{2}\\
&= \frac{1}{2\ell+1} \mathbb{E}\left \| a_{\ell, m} \otimes a_{\ell, m}  - \mathscr{F}_\ell  \right \|^2_{2}\\
&= \frac{1}{2\ell+1} \left ( \mathbb{E}\left \|a_{\ell, m}  \otimes a_{\ell, m}  \right \|^2_{2}  -  \|  \mathscr{F}_\ell  \|^2_{2} \right  )\\
&= \frac{1}{2\ell+1} \left ( \mathbb{E}\left \|a_{\ell, m}  \right \|^4_\hilbert  -  \|  \mathscr{F}_\ell  \|^2_{2}\right  ).
\end{align*}
Moreover, the fourth moment is given by
\begin{align*}
 \mathbb{E}\left \| a_{\ell, m} \right \|^4_\hilbert &=  \sum_{j=1}^\infty \sum_{j'=1}^\infty  \mathbb{E} \left[\langle a_{\ell,m},  e_{j;\ell} \rangle^2_\hilbert  \langle a_{\ell,m},  e_{j';\ell}\rangle^2_\hilbert  \right]\\
 &= \sum_{j} \mathbb{E}\langle a_{\ell,m},  e_{j;\ell} \rangle^4_\hilbert   +  \sum_{j \ne j'} \mathbb{E}\langle a_{\ell,m},  e_{j;\ell} \rangle^2_\hilbert \,\mathbb{E}\langle a_{\ell,m},  e_{j';\ell} \rangle^2_\hilbert   \\
 &= 3\sum_j \lambda^2_{j;\ell} + \left(\sum_{j} \lambda_{j;\ell} \right)^2  -  \sum_j \lambda^2_{j;\ell} \\
& = 2\left  \|  \mathscr{F}_\ell  \right \|^2_{2} + \mathcal{C}_\ell^2.
\end{align*}
In the end, we get
\begin{equation*}
\mathbb{E}\| \hat{\mathscr{F}}_\ell - \mathscr{F}_\ell \|^2_{2} =  \frac{\|\mathscr{F}_\ell \|^2_{2} + \mathcal{C}_\ell^2 }{2\ell+1},
\end{equation*}
which gives the final result.
\end{proof}

\subsection{Quantitative functional CLT for the power spectrum operators}

In the following we will present our main result, which consists in a functional \emph{quantitative} central limit theorem for the sample power spectrum operators, i.e., an upper bound for the $d_2$-distance between 
$$
\sqrt{\frac{2\ell+1}{\| \mathscr{F}_\ell\|^2_{2} + \mathcal{C}_\ell^2 }} (\hat{\mathscr{F}}_\ell - \mathscr{F}_\ell)
$$
and a target Gaussian $\hs^\dag$-valued random variable, at any given integer $\ell \ge 0$.

\begin{theorem}\label{theo::fqclt}
Assume that $\mathscr{F}_\ell$ is positive definite, for all integers $\ell \ge 0$, and write
$$
F_\ell =  \sqrt{\frac{2\ell+1}{\| \mathscr{F}_\ell\|^2_{2} + \mathcal{C}_\ell^2 }} (\hat{\mathscr{F}}_\ell - \mathscr{F}_\ell) \in \hs^\dag.
$$
Then, it holds that
\begin{align}\label{eq::maind2}
d_2( F_\ell, Z_\ell) &\le \frac{1+\sqrt{3}}{2\sqrt{3}}\sqrt{ \left ( \frac{12}{2\ell+1} + 3 \right)  \frac{12}{2\ell+1}},
\end{align}
where $Z_\ell$ are centered Gaussian $\hs^\dag$-valued random variables with the same covariance operators as the $F_\ell$, that is, $(\|\mathscr{F}_\ell\|^2_{2} + \mathcal{C}_\ell^2  )^{-1}\mathscr{S}_\ell$.
\end{theorem}

\begin{remark}
The terms $12/(2\ell+1) + 3$ and $12/(2\ell +1)$ appearing in Equation \eqref{eq::maind2} are precisely the fourth moment and cumulant that one can obtain in the real-valued case (see \cite[Section 8 -- page 199]{MP:11}). However, in the infinite dimensional setting, they represent a (tight) upper bound and the exact expressions for the fourth moment and cumulant are reported in Equation \eqref{eq::d2}.
\end{remark}

\begin{remark}
It is important to note that this result is non-asymptotic, which means that it holds for every finite integer $\ell \ge 0$. A careful inspection of the proof then reveals that a quantitative functional central limit theorem can also be derived for the sample covariance operator $\hat{\mathscr{R}}_n$ based on $X_1,\dots,X_n$ (zero mean) independent and identically distributed Gaussian $\hilbert$-valued random variables. It is already established that, if $\mathbb{E}\|X_1\|^4_\hilbert< \infty$ and $ \mathscr{R}= \mathbb{E}[X_1 \otimes X_1]$, $\sqrt{n}(\hat{\mathscr{R}}_n - \mathscr{R})$ converges in distribution in $\hs$ (and hence also in $\hs^\dag$) to some well defined Gaussian random variable (see \cite[Theorem 8.1.2]{Hsing}). Then, if we write
$$
F_n = \sqrt{\frac{n}{\| \mathscr{R}\|^2_{2} + \| \mathscr{R}\|^2_1}}(\hat{\mathscr{R}}_n - \mathscr{R}),
$$
we can get the same type of bound for $d_2(F_n, Z)$, where $Z$ is the limiting Gaussian $\hs^\dag$-valued random variable divided by $\sqrt{\| \mathscr{R}\|^2_{2} + \| \mathscr{R}\|^2_1}$.
However, in the high-frequency ergodicity setting we discussed, this result assumes much greater relevance. Indeed, even if it is not guaranteed that the limiting distribution is well defined (it depends on the behavior of $(\|\mathscr{F}_\ell\|^2_{2} + \mathcal{C}_\ell^2  )^{-1}\mathscr{S}_\ell$ as $\ell \to \infty$), we still have that for finite large enough $\ell$ the distribution is approximately Gaussian. 
\end{remark}


To prove the main theorem, we need first the following lemma which gives a characterization of the covariance operator of $\sqrt{2\ell+1}(\hat{\mathscr{F}}_\ell - \mathscr{F}_\ell)$.

\begin{lemma}\label{lemma::sl}
For any fixed integer $\ell \ge 0$, the covariance operator of $\sqrt{2\ell+1}(\hat{\mathscr{F}}_\ell - \mathscr{F}_\ell)$ on $\hs^\dag$ is given by
$$
\mathscr{S}_\ell  =2 \sum_{\substack{j,j'=1 \\ j \ge j'}}^\infty \lambda_{j;\ell} \lambda_{j';\ell}E_{j,j';\ell}  \otimes_2 E_{j,j';\ell},$$
where
\begin{equation*}\label{eq::Econs}
E_{j,j';\ell} = \begin{cases}
e_{j;\ell} \otimes  e_{j;\ell} & j=j'\\
\frac{e_{j;\ell} \otimes  e_{j';\ell}}{\sqrt{2}}  + \frac{e_{j';\ell} \otimes  e_{j;\ell}}{\sqrt{2}} &j \ne j'
\end{cases}.
\end{equation*}
If in addition $\mathscr{F}_\ell$ is positive definite, then also $\mathscr{S}_\ell$ is.
\end{lemma}

\begin{proof}[Proof of Lemma \ref{lemma::sl}]
We define the covariance operator $\mathscr{S}_\ell: \hs^\dag \to \hs^\dag$ as 
$$\mathscr{S}_\ell = (2\ell+1)\mathbb{E}[(\hat{\mathscr{F}}_\ell - \mathscr{F}_\ell) \otimes_2 (\hat{\mathscr{F}}_\ell - \mathscr{F}_\ell)] = \mathbb{E}[(a_{\ell,0} \otimes a_{\ell,0} - \mathscr{F}_\ell ) \otimes_2 (a_{\ell,0} \otimes a_{\ell,0} - \mathscr{F}_\ell) ].$$

Now, for $j,j'\in \mathbb{N}, \ j \ge j'$, define
$$
E_{j,j';\ell} = \begin{cases}
e_{j;\ell} \otimes  e_{j;\ell} & j=j'\\
\frac{e_{j;\ell} \otimes  e_{j';\ell}}{\sqrt{2}}  + \frac{e_{j';\ell} \otimes  e_{j;\ell}}{\sqrt{2}} &j \ne j'
\end{cases}.
$$
It is easy to check that they are orthonormal, i.e.,
$$
\langle E_{j,j';\ell} , E_{k,k';\ell}  \rangle_2 = \begin{cases}
1 &\text{when } j=k \text{ and } j'=k' \\
0 &\text{otherwise}
\end{cases}.
$$
Since it holds that
$$
a_{\ell,m} = \sum_{j=1}^\infty \langle a_{\ell,m}, e_{j;\ell} \rangle_\hilbert\,  e_{j;\ell},
$$
with probability one in $\hilbert$, then we can write
\begin{align}\label{eq::tensor}
a_{\ell,m} \otimes a_{\ell,m} - \mathscr{F}_\ell &=  \sum_{j=1}^\infty \sum_{j'=1}^\infty \left (\langle a_{\ell,m},  e_{j;\ell} \rangle_\hilbert  \langle a_{\ell,m},  e_{j';\ell} \rangle_\hilbert  -\lambda_{j;\ell} \delta_j^{j'}\right)  e_{j;\ell} \otimes  e_{j';\ell}\notag\\
&= \sum_{\substack{j,j'=1 \\ j \ge j'}}^\infty \tilde{a}_{\ell,m}(j,j') E_{j,j';\ell}, 
\end{align}
where
$$
\tilde{a}_{\ell,m}(j,j') = \begin{cases}
 \langle a_{\ell,m}, e_{j} \rangle_\hilbert^2  - \lambda_{j;\ell}&\text{when } j=j'\\
\sqrt{2}\langle a_{\ell,m}, e_{j} \rangle_\hilbert\langle a_{\ell,m}, e_{j'} \rangle_\hilbert &\text{otherwise }
\end{cases}.
$$
These coefficients are uncorrelated with zero-mean and $\mathbb{E}|\tilde{a}_{\ell,m}(j,j')|^2=2\lambda_{j;\ell} \lambda_{j';\ell}$. Then, we can derive
$$
\mathscr{S}_\ell  = 2 \sum_{\substack{j,j'=1 \\ j \ge j'}}^\infty \lambda_{j;\ell} \lambda_{j';\ell}E_{j,j';\ell}  \otimes_2 E_{j,j';\ell} .
$$

In addition, if $\mathscr{F}_\ell$ is positive definite, the set $\{E_{j,j';\ell} =,\ j,j'\in \mathbb{N}, \ j \ge j'\}$ is complete in $\hs^\dag$. Indeed, if $B \in \hs^\dag$ is orthogonal to $E_{j,j';\ell}$, for all $j,j'\in \mathbb{N}, \ j \ge j'$, then we have that
$$
0=\langle B, E_{j,j';\ell} \rangle_2  = \begin{cases} \langle Be_{j;\ell}, e_{j;\ell} \rangle_\hilbert &j=j' \\ \sqrt{2} \langle Be_{j;\ell}, e_{j';\ell} \rangle_\hilbert &j\ne j'
\end{cases},
$$
which implies that $B$ is the zero operator, since $\{e_{j;\ell}, \ j \in \mathbb{N}\}$ is a CONS for $\hilbert$.
As a consequence, if $\mathscr{F}_\ell$ is positive definite, then also $\mathscr{S}_\ell$ is. 

\end{proof}

\begin{proof}[Proof of Theorem \ref{theo::fqclt}]
Recall that the space $\randomel$ is isomorphic to $\randomvar \otimes \hilbert$.
Now, since $a_{\ell,m}$ is a centered Gaussian $\hilbert$-valued random variable and belongs to $\randomel$, we have that
$$
a_{\ell,m} = \sum_{j=1}^\infty \langle a_{\ell,m},  e_{j;\ell} \rangle_\hilbert   e_{j;\ell} \in  \mathcal{H}_1(\hilbert),
$$
i.e., it is an element of the first-order Wiener chaos of $L^2(\Omega;\hilbert)$. Note that $\langle a_{\ell,m},  e_{j;\ell} \rangle_\hilbert = H_1(\langle a_{\ell,m},  e_{j;\ell} \rangle_\hilbert )$. 
In addition, $A_{\ell, m}= (a_{\ell,m} \otimes a_{\ell, m} - \mathscr{F}_\ell )$ belongs to $L^2(\Omega; \hs^\dag)$ and its expansion is given in \eqref{eq::tensor}, with 
$$
\tilde{a}_{\ell,m}(j,j') = \begin{cases}
 \lambda_{j;\ell} H_2\left( \frac{\langle a_{\ell,m}, e_{j} \rangle_\hilbert}{\sqrt{\lambda_{j;\ell}}} \right)&\text{when } j=j'\\
\sqrt{2}H_1(\langle a_{\ell,m}, e_{j} \rangle_\hilbert)H_1(\langle a_{\ell,m}, e_{j'} \rangle_\hilbert) &\text{otherwise }
\end{cases}.
$$
This implies that
$
A_{\ell, m} \in \mathcal{H}_2(\hs^\dag),
$
i.e., it is an element of the second-order Wiener chaos of $L^2(\Omega; \hs^\dag)$, as well as $\hat{\mathscr{F}}_\ell - \mathscr{F}_\ell$.

Then, Lemma \ref{lemma::second-moment} gives
$$
\mathbb{E}\| \sqrt{2\ell+1} (\hat{\mathscr{F}}_\ell - \mathscr{F}_\ell) \|^2_{2} = \| \mathscr{F}_\ell\|^2_{2} + \mathcal{C}_\ell^2.
$$
Moreover, we can compute the following fourth moment
\begin{align*}
&\mathbb{E}\| \sqrt{2\ell+1} (\hat{\mathscr{F}}_\ell - \mathscr{F}_\ell) \|^4_{2} = \frac{1}{(2\ell+1)^2} \sum_{m_1,m_2,m_3,m_4} \mathbb{E} [ \langle A_{\ell,m_1},  A_{\ell,m_2} \rangle_{2} \langle A_{\ell,m_3},  A_{\ell,m_4} \rangle_{2}]\\
=&\frac{1}{(2\ell+1)^2} \sum_{m} \mathbb{E}  \| A_{\ell,m} \|^4_{2} + \frac{1}{(2\ell+1)^2} \sum_{m \ne m} \mathbb{E}  \| A_{\ell,m} \|^2_{2}\, \mathbb{E}  \| A_{\ell,m'} \|^2_{2}  + \frac{2}{(2\ell+1)^2} \sum_{m\ne m'} \mathbb{E}  \langle A_{\ell,m} , A_{\ell,m'} \rangle_2^2 \\
=&\frac{1}{(2\ell+1)^2} \sum_{m} \operatorname{Var} \|A_{\ell,m}\|^2_2 + \frac{1}{(2\ell+1)^2}  \left (\sum_{m} \mathbb{E}  \| A_{\ell,m} \|^2_{2} \right)^2  + \frac{2}{(2\ell+1)^2} \sum_{m\ne m'} \mathbb{E}  \langle A_{\ell,m} , A_{\ell,m'} \rangle_2^2 \\
=& \frac{1}{2\ell+1} \left( 40 \|\mathscr{F}_\ell\|_4^4 + 16 \| \mathscr{F}_\ell \|^4_{2}  \right) + \left (  \| \mathscr{F}_\ell\|^2_{2} + \mathcal{C}_\ell^2 \right)^2 + 4 \frac{2\ell}{2\ell+1}  \left (  \| \mathscr{F}_\ell\|^4_{4} +  \| \mathscr{F}_\ell\|^4_{2} \right),
\end{align*}
where in particular for the variance term we have used the fact that
\begin{align*}
\operatorname{Var} \|A_{\ell,m}\|^2_2 &= \sum_{\substack{j,j'=1 \\ j \ge j'}}^\infty \operatorname{Var} |\tilde{a}_{\ell,m}(j,j')|^2\\
&= \sum_{\substack{j,j'=1 \\ j \ge j'}}^\infty \mathbb{E} |\tilde{a}_{\ell,m}(j,j')|^4 -  \sum_{\substack{j,j'=1 \\ j \ge j'}}^\infty \left ( \mathbb{E} |\tilde{a}_{\ell,m}(j,j')|^2 \right)^2 \\
&= 60 \sum_{j=1}^\infty  \lambda_{j;\ell}^4 + 36\sum_{\substack{j,j'=1 \\ j > j'}}^\infty \lambda_{j;\ell}^2 \lambda_{j';\ell}^2 - 4 \sum_{j=1}^\infty  \lambda_{j;\ell}^4 - 4\sum_{\substack{j,j'=1 \\ j > j'}}^\infty \lambda_{j;\ell}^2 \lambda_{j';\ell}^2 \\
&= 40 \|\mathscr{F}_\ell\|^4_4 + 16 \|\mathscr{F}_\ell\|^4_2,
\end{align*}
see also Proposition 2.7.13, \cite{noupebook}. Hence,
\begin{align*}
&\mathbb{E}\| \sqrt{2\ell+1} (\hat{\mathscr{F}}_\ell - \mathscr{F}_\ell) \|^4_{2} = \frac{12}{2\ell+1} \left (  3 \|\mathscr{F}_\ell\|_4^4 +  \| \mathscr{F}_\ell \|^4_{2}  \right) + \left (  \| \mathscr{F}_\ell\|^2_{2} + \mathcal{C}_\ell^2 \right)^2 + 4  \left (  \| \mathscr{F}_\ell\|^4_{4} +  \| \mathscr{F}_\ell\|^4_{2} \right).
\end{align*}
Note that $2\left (\| \mathscr{F}_\ell\|^4_{4} +  \| \mathscr{F}_\ell\|^4_{2} \right)$ is the squared Hilbert-Schmidt norm of $\mathscr{S}_\ell$, while  $\| \mathscr{F}_\ell\|^2_{2} + \mathcal{C}_\ell^2$ is the trace norm of $\mathscr{S}_\ell$.

Now, write
$$
F_\ell =  \sqrt{\frac{2\ell+1}{\| \mathscr{F}_\ell\|^2_{2} + \mathcal{C}_\ell^2 }} (\hat{\mathscr{F}}_\ell - \mathscr{F}_\ell),
$$
and define $Z_\ell$ to be a centered, non-degenerate Gaussian $\hs^\dag$-valued random variable, with covariance operator $(\|\mathscr{F}_\ell\|^2_{2} + \mathcal{C}_\ell^2  )^{-1}\mathscr{S}_\ell$. See also \cite[Definition 2.1]{campese}.

By \cite[Theorem 3.10 (ii)]{campese}, the bound can be written as follows
\begin{align}\label{eq::d2}
d_2( F_\ell, Z_\ell ) &\le \frac{1+\sqrt{3}}{2\sqrt{3}}\sqrt{ \left ( \frac{12}{2\ell+1} \frac{3 \|\mathscr{F}_\ell\|_4^4 +  \| \mathscr{F}_\ell \|^4_{2}}{\left (\|\mathscr{F}_\ell\|^2_{2} + \mathcal{C}_\ell^2  \right)^2}   + 1  +  \frac{4 \left (  \| \mathscr{F}_\ell\|^4_{4} +  \| \mathscr{F}_\ell\|^4_{2} \right)}{\left (\|\mathscr{F}_\ell\|^2_{2} + \mathcal{C}_\ell^2  \right)^2} \right)  \frac{12}{2\ell+1} \frac{3 \|\mathscr{F}_\ell\|_4^4 +  \| \mathscr{F}_\ell \|^4_{2}}{\left (\|\mathscr{F}_\ell\|^2_{2} + \mathcal{C}_\ell^2  \right)^2}  } \notag\\
&\le  \frac{1+\sqrt{3}}{2\sqrt{3}}\sqrt{ \left ( \frac{12}{2\ell+1} + 3 \right)  \frac{12}{2\ell+1}}.
\end{align}
\end{proof}

\subsection{Quantitative CLT for the reduced power spectrum}

In this section, we will prove high-frequency asymptotics also for the reduced power spectrum $\{\mathcal{C}_\ell, \ell \ge 0 \}$.
To this end, for any integer $\ell \ge 0$, we define $\hat{\mathcal{C}}_\ell$ as 
$$
\hat{\mathcal{C}}_\ell = \frac{1}{2\ell+1} \sum_{m=-\ell}^\ell \|a_{\ell, m}\|_\hilbert^2,
$$
which is an unbiased estimator of $\mathcal{C}_\ell = \mathbb{E}\|a_{\ell,m}\|^2$. Then, we have the following result, which can be directly compared with \cite[Lemma 8.3]{MP:11}.

\begin{theorem}\label{theo::qclt}
For all integers $\ell \ge 0$,
\begin{equation}\label{eq::qclt}
d_{\operatorname{TV}} \left (\sqrt{\frac{2\ell+1}{2\|\mathscr{F}_\ell\|^2_{2}}} (\hat{\mathcal{C}}_\ell - \mathcal{C}_\ell ) , Z\right) \le \sqrt{\frac{8}{2\ell+1}},
\end{equation}
where $Z\sim\mathcal{N}(0,1)$.
\end{theorem}

\begin{proof}[Proof of Theorem \ref{theo::qclt}]
To prove the quantitative central limit theorem \eqref{eq::qclt}, without loss of generality, we assume that $\mathscr{F}_\ell$ is positive definite, for any integer $\ell \ge 0$. First of all, note that if $a_{\ell,m}$ is a Gaussian $\hilbert$-valued random variable, then $\langle a_{\ell,m},  e_{j;\ell} \rangle_\hilbert \sim \mathcal{N}(0, \lambda_{j;\ell})$ and the centered squared norm
$$
\|a_{\ell,m} \|^2_\hilbert - \mathcal{C}_\ell = \sum_{j=1}^\infty  \left ( \langle a_{\ell,m},  e_{j;\ell} \rangle^2_\hilbert - \lambda_{j;\ell} \right)
$$ 
is in the second-order Wiener chaos of $L^2(\Omega;\mathbb{R})$, as well as $\hat{\mathcal{C}}_\ell - \mathcal{C}_\ell$.

Then, we have that
$$
\mathbb{E} | \sqrt{2\ell+1}(\hat{\mathcal{C}}_\ell - \mathcal{C}_\ell) |^2 = \operatorname{Var} \|a_{\ell, m}\|_\hilbert^2 = \mathbb{E} \|a_{\ell, m}\|_\hilbert^4 - \left ( \mathbb{E}\|a_{\ell, m}\|_\hilbert^2\right)^2  = 2 \|\mathscr{F}_\ell \|^2_{2}.
$$
Moreover, we can compute the following 4-th cumulant
\begin{align*}
\operatorname{Cum}_4 [\sqrt{2\ell+1} (\hat{\mathcal{C}}_\ell - \mathcal{C}_\ell )  ] &= \frac{1}{(2\ell+1)^2}\operatorname{Cum}_4\left [ \sum_{m=-\ell}^\ell \left (\|a_{\ell, m}\|_\hilbert^2  - \mathcal{C}_\ell \right)\right]\\
&=\frac{1}{(2\ell+1)^2}  \sum_{m=-\ell}^\ell \operatorname{Cum}_4\left [ \|a_{\ell, m}\|_\hilbert^2  - \mathcal{C}_\ell  \right ]\\
&= \frac{1}{(2\ell+1)^2}  \sum_{m=-\ell}^\ell \operatorname{Cum}_4\left [ \sum_j \lambda_{j;\ell}  H_2\left( \frac{\langle a_{\ell,m}, e_{j} \rangle_\hilbert}{\sqrt{\lambda_{j;\ell}}} \right)\right]\\
&=\frac{48}{2\ell+1} \sum_{j} \lambda_{j;\ell}^4,
\end{align*}
see also Proposition 2.7.13, \cite{noupebook}. Hence,
$$
\operatorname{Cum}_4 \left[ \sqrt{\frac{2\ell+1}{2\|\mathscr{F}_\ell\|^2_{2}}} (\hat{\mathcal{C}}_\ell - \mathcal{C}_\ell )  \right] = \frac{12}{2\ell+1} \frac{\|\mathscr{F}_\ell\|_4^4}{\|\mathscr{F}_\ell\|_2^4} \le \frac{12}{2\ell+1} ,
$$
and we have the result.
\end{proof}

\section*{Appendix}

In this Appendix we state a general result on covariance and cross-covariance operators. This result is standard, but we decided to include it in the paper for completeness.

Suppose that we have two zero-mean $\hilbert$-valued random variables $T_1$ and $T_2$ defined on the same
probability space $(\Omega, \mathfrak{F}, \operatorname{P})$. Assume that, for $i = 1, 2$, $\mathbb{E}\|T_i\|_\hilbert^2 < \infty$ and consider the covariance and cross-covariance operators
$$
\mathscr{R}_{ij} = \int_\Omega T_i \otimes T_j\, d\operatorname{P}, \qquad i,j=1,2.
$$

\begin{lemma}\label{lemma::trace}
If $\mathscr{R}_{12}$ is self-adjoint, then $\mathscr{R}_{12}=\mathscr{R}_{21}$ and
$$
\|\mathscr{R}_{12}\|_1 \le \|\mathscr{R}_{11}\|^{1/2}_1 \|\mathscr{R}_{22}\|^{1/2}_1.
$$
\end{lemma}

\begin{proof}[Proof of Lemma \ref{lemma::trace}]
We know from \cite[Theorem 7.2.9 (3)]{Hsing} that the adjoint of $\mathscr{R}_{12}$ is $\mathscr{R}_{21}$, thus it follows that $\mathscr{R}_{12}=\mathscr{R}_{21}$. We also know that $\mathscr{R}_{12}$ belongs to $\hs$ and hence it is compact. The spectral theorem for compact and self-adjoint operators entails that there exists a CONS $\{e_j, j \in \mathbb{N}\}$ composed of eigenvectors of $\mathscr{R}_{12}$. Now, since $\mathscr{R}_{12}$ is self-adjoint we have that
$$
\|\mathscr{R}_{12}\|_1 = \sum_{j=1}^\infty |\langle \mathscr{R}_{12}e_j, e_j \rangle_\hilbert|.
$$
In addition, since $\{e_j, j \in \mathbb{N}\}$ is a CONS and the covariance operators $\mathscr{R}_{ii}, \ i=1,2,$ are nonnegative definite, we can write
$$
\|\mathscr{R}_{ii}\|_1 =  \sum_{j=1}^\infty \langle \mathscr{R}_{ii}e_j, e_j \rangle_\hilbert, \qquad i=1,2.
$$
Moreover, from \cite[Theorem 7.2.9 (2)]{Hsing} we have that
$$|\langle \mathscr{R}_{12}e_j, e_j \rangle_\hilbert| \le |\langle \mathscr{R}_{11}e_j, e_j \rangle_\hilbert|^{1/2} |\langle \mathscr{R}_{22}e_j, e_j \rangle_\hilbert|^{1/2}, \qquad j \in \mathbb{N},$$ 
and thus the result follows by applying Cauchy-Schwartz inequality.
\end{proof}

\bibliographystyle{imsart-number}
\bibliography{mybiblio}

\begin{thebibliography}{31}

\bibitem{bergporcu}
\begin{barticle}[author]
\bauthor{\bsnm{Berg},~\bfnm{C.}\binits{C.}} \AND
  \bauthor{\bsnm{Porcu},~\bfnm{E.}\binits{E.}}
(\byear{2017}).
\btitle{From {S}choenberg coefficients to {S}choenberg functions}.
\bjournal{Constructive Approximation}
\bvolume{45}
\bpages{217--241}.
\end{barticle}
\endbibitem

\bibitem{campese}
\begin{barticle}[author]
\bauthor{\bsnm{Bourguin},~\bfnm{Solesne}\binits{S.}} \AND
  \bauthor{\bsnm{Campese},~\bfnm{Simon}\binits{S.}}
(\byear{2020}).
\btitle{Approximation of {H}ilbert-valued {G}aussians on {D}irichlet
  structures}.
\bjournal{Electronic Journal of Probability}
\bvolume{25}
\bpages{1--30}.
\bdoi{10.1214/20-EJP551}
\end{barticle}
\endbibitem

\bibitem{brockwelldavis}
\begin{bbook}[author]
\bauthor{\bsnm{Brockwell},~\bfnm{P.~J.}\binits{P.~J.}} \AND
  \bauthor{\bsnm{Davis},~\bfnm{R.~A.}\binits{R.~A.}}
(\byear{1991}).
\btitle{Time Series: Theory and Methods},
\bedition{2nd} ed.
\bseries{Springer Series in Statistics}.
\bpublisher{Springer-Verlag}.
\end{bbook}
\endbibitem

\bibitem{CM:18b}
\begin{barticle}[author]
\bauthor{\bsnm{Cammarota},~\bfnm{V.}\binits{V.}} \AND
  \bauthor{\bsnm{Marinucci},~\bfnm{D.}\binits{D.}}
(\byear{2018}).
\btitle{A quantitative central limit theorem for the {Euler}-{Poincar\'{e}}
  characteristic of random spherical eigenfunctions}.
\bjournal{The Annals of Probability}
\bvolume{46}
\bpages{3188--3288}.
\end{barticle}
\endbibitem

\bibitem{SPHARMA}
\begin{barticle}[author]
\bauthor{\bsnm{Caponera},~\bfnm{A.}\binits{A.}}
(\byear{2021}).
\btitle{{SPHARMA} approximations for stationary functional time series on the
  sphere}.
\bjournal{Statistical Inference for Stochastic Processes}
\bvolume{24}
\bpages{609--634}.
\end{barticle}
\endbibitem

\bibitem{cdv19}
\begin{barticle}[author]
\bauthor{\bsnm{Caponera},~\bfnm{Alessia}\binits{A.}},
  \bauthor{\bsnm{Durastanti},~\bfnm{Claudio}\binits{C.}} \AND
  \bauthor{\bsnm{Vidotto},~\bfnm{Anna}\binits{A.}}
(\byear{2021}).
\btitle{{LASSO} estimation for spherical autoregressive processes}.
\bjournal{Stochastic Processes and their Applications}
\bvolume{137}
\bpages{167--199}.
\end{barticle}
\endbibitem

\bibitem{CFSP:22}
\begin{barticle}[author]
\bauthor{\bsnm{Caponera},~\bfnm{Alessia}\binits{A.}},
  \bauthor{\bsnm{Fageot},~\bfnm{Julien}\binits{J.}},
  \bauthor{\bsnm{Simeoni},~\bfnm{Matthieu}\binits{M.}} \AND
  \bauthor{\bsnm{Panaretos},~\bfnm{Victor~M.}\binits{V.~M.}}
(\byear{2022}).
\btitle{{Functional estimation of anisotropic covariance and autocovariance
  operators on the sphere}}.
\bjournal{Electronic Journal of Statistics}
\bvolume{16}
\bpages{5080--5148}.
\bdoi{10.1214/22-EJS2064}
\end{barticle}
\endbibitem

\bibitem{cm19}
\begin{barticle}[author]
\bauthor{\bsnm{Caponera},~\bfnm{Alessia}\binits{A.}} \AND
  \bauthor{\bsnm{Marinucci},~\bfnm{Domenico}\binits{D.}}
(\byear{2021}).
\btitle{{Asymptotics for spherical functional autoregressions}}.
\bjournal{The Annals of Statistics}
\bvolume{49}
\bpages{346--369}.
\end{barticle}
\endbibitem

\bibitem{CGR:22}
\begin{barticle}[author]
\bauthor{\bsnm{Caramellino},~\bfnm{Lucia}\binits{L.}},
  \bauthor{\bsnm{Giorgio},~\bfnm{Giacomo}\binits{G.}} \AND
  \bauthor{\bsnm{Rossi},~\bfnm{Maurizia}\binits{M.}}
(\byear{2022}).
\btitle{{Convergence in Total Variation for nonlinear functionals of random
  hyperspherical harmonics}}.
\bjournal{arXiv:2206.02605}.
\end{barticle}
\endbibitem

\bibitem{DRD:astro}
\begin{barticle}[author]
\bauthor{\bsnm{Dick},~\bfnm{Jason}\binits{J.}},
  \bauthor{\bsnm{Remazeilles},~\bfnm{Mathieu}\binits{M.}} \AND
  \bauthor{\bsnm{Delabrouille},~\bfnm{Jacques}\binits{J.}}
(\byear{2010}).
\btitle{{Impact of calibration errors on CMB component separation using FastICA
  and ILC}}.
\bjournal{Monthly Notices of the Royal Astronomical Society}
\bvolume{401}
\bpages{1602--1612}.
\bdoi{10.1111/j.1365-2966.2009.15798.x}
\end{barticle}
\endbibitem

\bibitem{gneiting2}
\begin{barticle}[author]
\bauthor{\bsnm{Gneiting},~\bfnm{Tilmann}\binits{T.}}
(\byear{2013}).
\btitle{Strictly and non-strictly positive definite functions on spheres}.
\bjournal{Bernoulli}
\bvolume{19}
\bpages{1327--1349}.
\end{barticle}
\endbibitem

\bibitem{Hsing}
\begin{bbook}[author]
\bauthor{\bsnm{Hsing},~\bfnm{Tailen}\binits{T.}} \AND
  \bauthor{\bsnm{Eubank},~\bfnm{Randall}\binits{R.}}
(\byear{2015}).
\btitle{Theoretical Foundations of Functional Data Analysis, with an
  Introduction to Linear Operators}.
\bpublisher{John Wiley \& Sons}.
\end{bbook}
\endbibitem

\bibitem{langschwab}
\begin{barticle}[author]
\bauthor{\bsnm{Lang},~\bfnm{A.}\binits{A.}} \AND
  \bauthor{\bsnm{Schwab},~\bfnm{C.}\binits{C.}}
(\byear{2015}).
\btitle{Isotropic {G}aussian random fields on the sphere: Regularity, fast
  simulation and stochastic partial differential equations}.
\bjournal{The Annals of Applied Probability}
\bvolume{25}
\bpages{3047--3094}.
\end{barticle}
\endbibitem

\bibitem{MP:ergodicity}
\begin{barticle}[author]
\bauthor{\bsnm{Marinucci},~\bfnm{Domenico}\binits{D.}} \AND
  \bauthor{\bsnm{Peccati},~\bfnm{Giovanni}\binits{G.}}
(\byear{2010}).
\btitle{Ergodicity and Gaussianity for spherical random fields}.
\bjournal{Journal of Mathematical Physics}
\bvolume{51}
\bpages{043301}.
\bdoi{10.1063/1.3329423}
\end{barticle}
\endbibitem

\bibitem{MP:11}
\begin{bbook}[author]
\bauthor{\bsnm{Marinucci},~\bfnm{D.}\binits{D.}} \AND
  \bauthor{\bsnm{Peccati},~\bfnm{G.}\binits{G.}}
(\byear{2011}).
\btitle{Random Fields on the Sphere: Representation, Limit Theorems and
  Cosmological Applications}.
\bseries{London Mathematical Society Lecture Note Series}.
\bpublisher{Cambridge University Press}.
\end{bbook}
\endbibitem

\bibitem{marinucci2013mean}
\begin{barticle}[author]
\bauthor{\bsnm{Marinucci},~\bfnm{Domenico}\binits{D.}} \AND
  \bauthor{\bsnm{Peccati},~\bfnm{Giovanni}\binits{G.}}
(\byear{2013}).
\btitle{Mean-square continuity on homogeneous spaces of compact groups}.
\bjournal{Electronic Communications in Probability}
\bvolume{18}
\bpages{1--10}.
\end{barticle}
\endbibitem

\bibitem{MR:15}
\begin{barticle}[author]
\bauthor{\bsnm{Marinucci},~\bfnm{Domenico}\binits{D.}} \AND
  \bauthor{\bsnm{Rossi},~\bfnm{Maurizia}\binits{M.}}
(\byear{2015}).
\btitle{Stein–Malliavin approximations for nonlinear functionals of random
  eigenfunctions on $\mathbb{S}^d$}.
\bjournal{Journal of Functional Analysis}
\bvolume{268}
\bpages{2379--2420}.
\bdoi{https://doi.org/10.1016/j.jfa.2015.02.004}
\end{barticle}
\endbibitem

\bibitem{MRW:20}
\begin{barticle}[author]
\bauthor{\bsnm{Marinucci},~\bfnm{D.}\binits{D.}},
  \bauthor{\bsnm{Rossi},~\bfnm{M.}\binits{M.}} \AND
  \bauthor{\bsnm{Wigman},~\bfnm{I.}\binits{I.}}
(\byear{2020}).
\btitle{The asymptotic equivalence of the sample trispectrum and the nodal
  length for random spherical harmonics}.
\bjournal{Annales de l'Institut Henri Poincar\'{e} Probabilit\'{e}s et
  Statistiques}
\bvolume{56}
\bpages{374--390}.
\end{barticle}
\endbibitem

\bibitem{MW:11}
\begin{barticle}[author]
\bauthor{\bsnm{Marinucci},~\bfnm{Domenico}\binits{D.}} \AND
  \bauthor{\bsnm{Wigman},~\bfnm{Igor}\binits{I.}}
(\byear{2011}).
\btitle{The defect variance of random spherical harmonics}.
\bjournal{Journal of Physics A: Mathematical and Theoretical}
\bvolume{44}
\bpages{355206}.
\bdoi{10.1088/1751-8113/44/35/355206}
\end{barticle}
\endbibitem

\bibitem{MW:14}
\begin{barticle}[author]
\bauthor{\bsnm{Marinucci},~\bfnm{Domenico}\binits{D.}} \AND
  \bauthor{\bsnm{Wigman},~\bfnm{Igor}\binits{I.}}
(\byear{2014}).
\btitle{On nonlinear functionals of random spherical eigenfunctions}.
\bjournal{Communications in Mathematical Physics}
\bvolume{327}
\bpages{849–872}.
\bdoi{10.1007/s00220-014-1939-7}
\end{barticle}
\endbibitem

\bibitem{noupebook}
\begin{bbook}[author]
\bauthor{\bsnm{Nourdin},~\bfnm{I.}\binits{I.}} \AND
  \bauthor{\bsnm{Peccati},~\bfnm{G.}\binits{G.}}
(\byear{2012}).
\btitle{Normal Approximations with Malliavin Calculus: from Stein's Method to
  Universality}.
\bpublisher{Cambridge University Press}.
\end{bbook}
\endbibitem

\bibitem{Panaretos2}
\begin{barticle}[author]
\bauthor{\bsnm{Panaretos},~\bfnm{Victor~M}\binits{V.~M.}} \AND
  \bauthor{\bsnm{Tavakoli},~\bfnm{Shahin}\binits{S.}}
(\byear{2013}).
\btitle{Fourier analysis of stationary time series in function space}.
\bjournal{The Annals of Statistics}
\bvolume{41}
\bpages{568--603}.
\end{barticle}
\endbibitem

\bibitem{Panaretos}
\begin{barticle}[author]
\bauthor{\bsnm{Panaretos},~\bfnm{Victor~M}\binits{V.~M.}} \AND
  \bauthor{\bsnm{Tavakoli},~\bfnm{Shahin}\binits{S.}}
(\byear{2013}).
\btitle{Cram{\'e}r--{K}arhunen--{L}o{\`e}ve representation and harmonic
  principal component analysis of functional time series}.
\bjournal{Stochastic Processes and their Applications}
\bvolume{123}
\bpages{2779--2807}.
\end{barticle}
\endbibitem

\bibitem{PFN:20}
\begin{barticle}[author]
\bauthor{\bsnm{Porcu},~\bfnm{Emilio}\binits{E.}},
  \bauthor{\bsnm{Furrer},~\bfnm{Reinhard}\binits{R.}} \AND
  \bauthor{\bsnm{Nychka},~\bfnm{Douglas}\binits{D.}}
(\byear{2020}).
\btitle{30 years of space-time covariance functions}.
\bjournal{WIREs Computational Statistics}
\bpages{e1512}.
\end{barticle}
\endbibitem

\bibitem{Ro:19}
\begin{barticle}[author]
\bauthor{\bsnm{Rossi},~\bfnm{Maurizia}\binits{M.}}
(\byear{2019}).
\btitle{The defect of random hyperspherical harmonics}.
\bjournal{Journal of Theoretical Probability}
\bvolume{32}
\bpages{2135--2165}.
\bdoi{10.1007/s10959-018-0849-6}
\end{barticle}
\endbibitem

\bibitem{schoenberg}
\begin{barticle}[author]
\bauthor{\bsnm{Schoenberg},~\bfnm{I.~J.}\binits{I.~J.}}
(\byear{1942}).
\btitle{Positive definite functions on spheres}.
\bjournal{Duke Mathematical Journal}
\bvolume{9}
\bpages{96--108}.
\end{barticle}
\endbibitem

\bibitem{szego}
\begin{bbook}[author]
\bauthor{\bsnm{Szeg\"{o}},~\bfnm{G.}\binits{G.}}
(\byear{1975}).
\btitle{Orthogonal Polynomials},
\bedition{4th} ed.
\bseries{American Mathematical Society Colloquium Publications}
\bvolume{23}.
\bpublisher{American Mathematical Society}.
\end{bbook}
\endbibitem

\bibitem{To:19}
\begin{barticle}[author]
\bauthor{\bsnm{Todino},~\bfnm{Anna~Paola}\binits{A.~P.}}
(\byear{2019}).
\btitle{A quantitative central limit theorem for the excursion area of random
  spherical harmonics over subdomains of $\mathbb{S}^2$}.
\bjournal{Journal of Mathematical Physics}
\bvolume{60}
\bpages{023505}.
\end{barticle}
\endbibitem

\bibitem{To:20}
\begin{barticle}[author]
\bauthor{\bsnm{Todino},~\bfnm{Anna~Paola}\binits{A.~P.}}
(\byear{2020}).
\btitle{{Nodal lengths in shrinking domains for random eigenfunctions on
  $\mathbb{S}^{2}$}}.
\bjournal{Bernoulli}
\bvolume{26}
\bpages{3081--3110}.
\bdoi{10.3150/20-BEJ1216}
\end{barticle}
\endbibitem

\bibitem{Zi:17}
\begin{barticle}[author]
\bauthor{\bsnm{Trübner},~\bfnm{Mara}\binits{M.}} \AND
  \bauthor{\bsnm{Ziegel},~\bfnm{Johanna}\binits{J.}}
(\byear{2017}).
\btitle{Derivatives of isotropic positive definite functions on spheres}.
\bjournal{Proceedings of the American Mathematical Society}
\bvolume{145}
\bpages{3017--3031}.
\bdoi{10.1090/proc/13561}
\end{barticle}
\endbibitem

\bibitem{Zi:14}
\begin{barticle}[author]
\bauthor{\bsnm{Ziegel},~\bfnm{Johanna}\binits{J.}}
(\byear{2014}).
\btitle{Convolution roots and differentiability of isotropic positive definite
  functions on spheres}.
\bjournal{Proceedings of the American Mathematical Society}
\bvolume{142}
\bpages{2063--2077}.
\end{barticle}
\endbibitem

\end{thebibliography}
\end{document}